\documentclass[11pt]{amsart}
\usepackage[english]{babel}
\usepackage{physics}
\usepackage{braket}
\usepackage{dsfont}
\usepackage[table,xcdraw]{xcolor}
\usepackage{multirow}
\usepackage{graphicx}
\usepackage{tikz}
\usepackage{amssymb}
\usepackage{mathtools}
\usetikzlibrary{matrix}

\usepackage[margin=1in]{geometry}
\usepackage[colorlinks = true, linkcolor = blue, urlcolor  = blue, citecolor = red]{hyperref}
\usepackage[utf8x]{inputenc}

\renewcommand{\epsilon}{\varepsilon}
\newcommand{\R}{\mathbb{R}}

\newcommand{\N}{\mathbb{N}}

\newcommand{\ds}{\displaystyle}
\newcommand{\id}{\bold{1}}
\newcommand{\Lp}{L^{p(\cdot)}}
\newcommand{\Lq}{L^{q(\cdot)}}

\def\ov{\overset}
\def\un{\underset}

\DeclareMathOperator*{\esssup}{ess\,sup}

\newtheorem{theorem}{Theorem}[section]
\newtheorem*{theorem*}{Theorem}
\newtheorem{lemma}[theorem]{Lemma}
\newtheorem{cor}[theorem]{Corollary}
\newtheorem*{cor*}{Corollary}
\newtheorem{ex}[theorem]{Example}
\newtheorem{defi}[theorem]{Definition}
\newtheorem{prop}[theorem]{Proposition}
\newtheorem{remark}[theorem]{Remark}

\mathtoolsset{showonlyrefs}

\begin{document}

\author{{\'A}ngela Capel}
\email{angela.capel@ma.tum.de}
\address{Department of Mathematics, Technische Universit\"at M\"unchen, 85748 Garching, Germany and Munich Center for Quantum Science and Technology (MCQST), M\"unchen, Germany}

\author{Jes{\'u}s Oc{\'a}riz}
\email{jesus.ocariz@uam.es}
\address{Departamento de Matem{\'a}ticas, Universidad Aut{\'o}noma de Madrid, Spain and Instituto de Ciencias Matem{\'a}ticas (CSIC-UAM-UC3M-UCM),  Madrid, Spain }

\title[Approximation with Neural Networks in Variable Lebesgue Spaces]{Approximation with Neural Networks \\in Variable Lebesgue Spaces}

\begin{abstract}
This paper concerns the universal approximation property with neural networks in variable Lebesgue spaces. We show that, whenever the exponent function of the space is bounded, every function can be approximated with shallow neural networks with any desired accuracy. This result subsequently leads to determine the universality of the approximation depending  on the boundedness of the exponent function. Furthermore,  whenever the exponent is unbounded, we obtain some characterization results for the subspace of functions that can be approximated.
\end{abstract}

\date{\today}

\maketitle

\vspace{-0.3cm}

\tableofcontents

\section{Introduction}\label{sec:intro}

Artificial neural networks are a model created with the purpose of imitating the behavior of  biological neural networks using digital computing. Their origins are tied back to \cite{mcculloch1943origin} and \cite{rosenblatt1958perceptron}, and since then numerous applications have been found  in a wide range of fields, varying from machine learning to computer vision, speech recognition or mathematical finances, among many others. A major problem in the theory of neural networks is that of approximating within a desired accuracy a generic class of functions using neural networks, initially motivated by the behavior for neural networks observed in the Representation Theorem due to Arnold  \cite{arnold1957threevariables} and Kolmogorov \cite{kolmogorov1957representation} and the aim of providing a theoretical justification for it.

The starting point of approximation theory for neural networks was the \textit{Universal Approximation Theorem} of \cite{cybenko1989approximation} and \cite{hornik1989multilayer}, which shows that every continuous function on a compact can be uniformly approximated by shallow neural networks with a continuous, non-polynomial activation function. Subsequent extensions of this result addressed the analogous problem for Lebesgue spaces with a finite exponent \cite{hornik1991approximation} and locally integrable spaces \cite{park1991universal}, meanwhile some others also considered the derivatives of the neural networks to show that shallow neural networks with a sufficiently smooth activation function and unrestricted width are dense in the space of sufficiently differentiable functions \cite{pinkus1999MLPmodel}. Soon after this wave of results for shallow neural networks, various surveys on the topic appeared in the literature, such as \cite{pinkus1997ridge}, \cite{scarselli1998survey}, \cite{tikk2003survey} and \cite{sanguineti2008survey}.

In the last years, many directions have been explored in the approximation theory for both shallow and deep neural networks. For neural networks with ReLU activation functions, there are a number of recent papers concerning various topics: Approximation for Besov spaces  \cite{suzuki2019besov},  regression \cite{schmidt2017regression} and optimization \cite{giulini2009optimization} problems, restriction to encodable weights \cite{petersen2018deepReLU},  negative results of approximation  \cite{almira2018negative}, estimates for the errors obtained in the approximation  \cite{petrushev1999ridge}, \cite{yarotsky2017errorRELU}, or,  in general,  deep neural networks  \cite{kidger2019UAdeep}, \cite{shaman2018deep}, among many others.  A simpler and inspiring new proof for the universality theory in deep neural networks can be found in \cite{heinecke2020ReLU}. Moreover, in \cite{hwang2020unrectifying} the authors develop a new technique named ``un-rectifying'' which transfers piece-wise continuous non-linear activation functions into piece-wise continuous linear functions and then use it to show that ReLU networks and MaxLU networks are indeed deep trees.

Furthermore, for more regular activation functions there are also numerous recent articles, namely \cite{barron1994approx}, \cite{bolcskei2019sparselydeep}, \cite{lin2019deepnets}, \cite{li2020highdimensions}, \cite{mhaskar1996NNoptimalapprox},  \cite{ohn2019smoothapprox}, \cite{tang2019chebyshev}, most of which focus on deep neural networks. Some other directions are currently being studied for the  problem of approximation too, such as the topological approach presented in \cite{kratsios2019approximation}, the application of these results to finding solutions of partial differential equations \cite{guhring2020encodableweights} or the comparison to approximation with tensor networks \cite{ali2020tensornetworks}.

In  this paper, we take a step forward in the theory of approximation with neural networks and address the \textbf{problem of approximating any function in a variable Lebesgue space} with enough precision using  shallow neural networks with various activation functions. Variable Lebesgue spaces are in particular locally integrable spaces. Even though there exist some previous results of universal approximation in locally integrable spaces, they depend on a metric defined in a way that the behavior of a function away from a compact is always negligible. Moreover, since that distance is not constructed from a norm, the approximations are not stable by dilations. In this manuscript, we overcome this situation and prove our approximation results employing the distance determined by the  usual norm of the space. 

Variable Lebesgue spaces are a generalization of Lebesgue spaces that might contain functions which do not belong to any Lebesgue space \cite{cruzuribe2013variablelp}. In particular, a variable Lebesgue space might include all the bounded functions even if the domain is not compact. Therefore, a natural problem that arises in this setting is that of approximating functions in non-compact domains. For example, continuous functions defined on  an unbounded interval with a suitable asymptotic limit. This type of functions may appear as the representation of a quantity  that follows a diffusion process with time (like the temperature at a certain point in a closed system).

More specifically, in this paper we show approximation results with neural networks for variable Lebesgue spaces depending on the boundedness of their exponent function. If the exponent of the space is essentially bounded, we show that a result of universal approximation holds, yielding thus an analogous behavior to that of usual Lebesgue spaces. On the other hand, if the exponent is unbounded, the situation is much more subtle, but we can characterize in some cases the subspace of functions which can be approximated with neural networks relying on some results of  \cite{amenta2019variablelp}. We first address the simpler case of variable sequence spaces to subsequently lift our results to a more general domain. Our results hold for most of the activation functions present in the literature, namely  any sigmoidal function (logistic sigmoid, hyperbolic tangent, Heaviside function, etc) or the rectifier function. 

The outline of the manuscript is the following: In Section \ref{sec:mainresults}, we present an informal exposition of the main results of the present article. Important notions and results on neural networks and variable Lebesgue spaces are reviewed in Section \ref{sec:prelim}. In Section \ref{sec:UA}, we collect some of the previous results on universal approximation in certain function spaces and provide several improvements or generalizations for them. Finally, in Section \ref{sec:approxvarLp}, we present our results on approximation with neural networks in variable Lebesgue spaces.

\section{Main results}\label{sec:mainresults}

A \textit{shallow neural network} is described by a function $g:\R^d\to \R$ given by
\begin{equation}\label{eq:neuralnetwork}
g(x)= \underset{j=1}{\overset{M}{\sum}} \alpha_j \, \sigma (w_j \cdot x + b_j) \,  ,
\end{equation}
where  $x \in \mathbb{R}^d$ represents the \emph{input} to the neural network, $g(x)\in \mathbb{R}$ the \emph{output}, $w_j \in \mathbb{R}^d$ and $\alpha_j\in \R$ are the \emph{weights} between first and second layer, and second and third layer, respectively, $b_j \in \mathbb{R}$ are the \emph{biases}, $\sigma: \mathbb{R} \rightarrow \mathbb{R}$ is the \emph{activation function} and $M$ the \textit{height}. The subspace generated by such functions will be denoted by $H_\sigma$.

Given an activation function $\sigma$ and a function normed space $(X,\|\cdot\|)$, the \textit{Universal Approximation (UA)} property for shallow neural networks can be formally stated as follows: 

\vspace{0.2cm}
\begin{center}
$\boxed{\text{ For every $ f\in X$ and $\forall \varepsilon >0$, there is a function $g \in H_\sigma$  such that $\|f-g\|<\varepsilon$. }}$
\end{center}

\vspace{0.2cm} 

The main results of this article concern the UA property for variable Lebesgue spaces, i.e. spaces of the form
$$
\Lp(\Omega):=\{f:\Omega\to \R: f \text{ measurable and } \| f\|_{p(\cdot)}<+\infty\} \, ,
$$
for an open $\Omega \subseteq \mathbb{R}^d$, where $p:\Omega\to [1,+\infty)$ is an exponent function and the norm is given by 
$$
\| f\|_{p(\cdot)} := \inf \left\{\lambda>0: \int_{\Omega}\left(\frac{\abs{f(x)}}{\lambda}\right)^{p(x)} \mathrm{d}x\leq 1  \right\} \, .
$$

The first of these results (which appears in the main text as Theorem \ref{thm:UALpVariable}) shows that  universal approximation  holds whenever the exponent function of the space is bounded.

\begin{theorem*}
Let $\Omega\subseteq \R^d$, consider $p: \Omega \rightarrow [1, + \infty)$ a bounded exponent function and $\sigma: \mathbb{R} \rightarrow \mathbb{R}$  discriminatory for $\Lp(K)$ for every compact $K\subset \Omega$. Then, truncated finite sums of the form
$$
g(x)=\begin{cases}
\displaystyle \sum_{j=1}^M \alpha_j \, \sigma(w_j\cdot x+b_j) & x\in K \, , \\
0 & x\in \Omega \setminus K \, ,
\end{cases}
$$
 with $K$ compact  are dense in $\Lp(\Omega)$.
\end{theorem*}

The condition imposed on $\sigma$ essentially means that if a functional over $\Lp(K)$ vanishes on $H_\sigma$, then the functional is identically null. Most of the activation functions considered in the literature verify this property, such as any continuous function different to an algebraic polynomial, for example. 

Whenever the exponent function of a variable Lebesgue space is unbounded, the situation is much more complex. Indeed, as a consequence of the separability of the space and the previous theorem, we provide a characterization of universal approximation in terms of the boundedness of the exponent function, which appears in the main text as Corollary \ref{cor:characterization}. 

\begin{cor*}
Let $\Omega\subseteq \R^d$, $p: \Omega \rightarrow [1, + \infty)$ a exponent function and $\sigma: \mathbb{R} \rightarrow \mathbb{R}$  continuous and with finite limit at $+ \infty$ and $- \infty$. Then,  UA holds for $\Lp(\Omega)$ if, and only if, $p$ is bounded.
\end{cor*}

Subsequently, since for unbounded exponent functions the UA fails, we study conditions to describe the subspace of functions which can be approximated with neural networks. In the following result, the condition of characterization is a generalization of the concept of having an asymptotic limit. 

\begin{theorem*}
Let $\Omega \subseteq \mathbb{R}$ be an unbounded interval and $p: \Omega \rightarrow [1, + \infty)$ an unbounded exponent function such that $L^\infty(\Omega)\subset \Lp(\Omega)$ and it is bounded in every compact subset of $\Omega$. Let $\sigma\in L^\infty(\R)$ be a non-constant, sigmoidal activation function. Then, the following conditions are equivalent for $f\in \Lp(\Omega)$:
\begin{enumerate}
\item For every $ \varepsilon>0$, there is a $g_\varepsilon \in H_\sigma$ such that $\|f-g_\varepsilon\|_{\Lp(\Omega)}<\varepsilon$.
\item There is a scalar $\beta \in \R$ such that
$$
\|[f-\beta \id_{\Omega}]\|_Q=0 \, ,
$$
where $\|\cdot\|_Q$ is the quotient norm given in Definition \ref{def:LpQ}.
\end{enumerate}
\end{theorem*}

This theorem appears as Theorem \ref{thm:UApnotbounded} in the main text.  In a nutshell, the quotient norm encodes the behavior of the function at $\infty$. Therefore, the functions that can be approximated are those which converge, in some sense, at $\infty$.

Finally, prior to these results we recall and adapt some classical results of universal approximation for some function spaces. More specifically, we address the space of continuous functions over a compact and locally integrable spaces. For the former, we prove an extension of the original result of universal approximation from continuous functions on a compact to the space of functions which vanish at infinity, in the unidimensional case, and a negative result of universal approximation, whereas for the latter we extend a result for radial activation functions to non-radial ones.  Moreover, even though variable Lebesgue spaces are locally integrable spaces, our contribution for those spaces is a major improvement, because our concept of distance comes from a norm and, therefore, approximations in this context are more interesting, since for example we can control the error of dilation.

\section{Preliminaries}\label{sec:prelim}

\subsection{Neural networks}\label{subsec:neuralnetworks}

In this subsection, we introduce the concepts and properties associated to neural networks that we will need for the rest of the paper. Some references for the mathematical formulation of neural networks in this context are \cite{haykin1998neuralnetworks}, \cite{bishop2006patternrecognition}, \cite{krosesmagt1993introneuralnetworks} or \cite{rojas1996neuralnetworks}, among many others.

In this paper, we focus on \textit{feedforward artificial neural networks} (denoted ANN hereafter), the simplest model for neural networks. A feedforward artificial neural network contains several nodes, arranged in layers serving differently depending on their position, namely \textit{input}, \textit{hidden} and \textit{output} nodes. Moreover, this type of neural networks only has a single input layer, which provides information from the environment to the network, and a single output layer, which transmits information from the network to the environment.

 In the past, different classes of ANN have been considered to show approximation in generic classes of functions. It is known that ANN with no hidden layer are not capable of approximating generic, non-linear, continuous functions \cite{widrow1990neuralnetworks}. On the opposite side,   ANN with two or more hidden layers, known as \textit{deep neural networks}, have given rise to a broad field of research in the past years. However, for simplicity we focus specifically in the case of  one hidden layer, commonly known as \textit{shallow neural networks}, for which typical results on universal approximation have been studied in the past. 

\begin{defi}
A \emph{feedforward shallow artificial neural network} (ANN) can be described by a finite linear combination of the form
\begin{equation*}
g(x) = \underset{j=1}{\overset{M}{\sum}} \alpha_j \, \sigma (w_j \cdot x + b_j) \, ,
\end{equation*}
where  $x \in \mathbb{R}^d$ represents the \emph{input} to the neural network, $g(x)\in \mathbb{R}$ the \emph{output}, $w_j \in \mathbb{R}^d$ and $\alpha_j\in \R$ are the \emph{weights} between first and second layer, and second and third layer, respectively, $b_j \in \mathbb{R}$ are the \emph{biases}, $\sigma: \mathbb{R} \rightarrow \mathbb{R}$ is the \emph{activation function} and $M$ the \emph{height}. 
\end{defi}

The subspace of all functions that can be obtained using an ANN with activation function $\sigma$ will be denoted hereafter by
$$
H_\sigma :=\left\{ g(x)=\sum_{j=1}^M  \alpha_j \, \sigma(w_j\cdot x +b_j)\right\} \, ,
$$
where $w_j, x \in \mathbb{R}^d$, $\alpha_j, b_j\in \R$ and $M \in \N$. Throughout the text, we will impose different conditions on $\sigma$ to obtain diverse results, since the role of the activation function is essential in the results of approximation of generic function spaces. Here we summarize some of the most typical examples for activation functions.

\begin{ex}\label{ex:activationfunctions}
Examples of activation functions:

\begin{itemize}
\item The Heaviside step function, which is piece-wise defined over  $\R$ in the following way:
\[ 
\sigma (x)=\begin{cases} 
      1  & \text{if } x>0 \, , \\
     0  & \text{if } x\leq 0 \, .
   \end{cases}
\]

\item The rectifier function (ReLU) is defined as:
\[ 
\sigma (x)=\begin{cases} 
       x & \text{if } x>0 \, , \\
       0  & \text{if } x\leq 0 \, .
   \end{cases}
\]

\item The logistic sigmoid is given by:
\[ 
\sigma (x)= \frac{1}{1+\operatorname{e}^{-x}} \, .
\] 

\item The hyperbolic tangent is an affinely transformed logistic sigmoid. It is given by:
\[ 
\sigma (x)= \operatorname{tanh}(x) \, .
\]

\end{itemize}

\end{ex}

For further examples of activation functions, see  \cite{scarselli1998survey} for instance.  Now, we introduce a property concerning the activation function which will appear often throughout the rest of the text. 

\begin{defi}\label{def:sigmoidal}
Given an activation function $\sigma : \mathbb{R} \rightarrow \mathbb{R}$, we say that $\sigma $ is \emph{sigmoidal} if 
\begin{equation*}
\sigma (t)=\begin{cases} 
       c_{+\infty}  & \text{as } t \rightarrow + \infty \, , \\
       c_{-\infty}  & \text{as }  t \rightarrow - \infty \, ,
   \end{cases}
\end{equation*}
where the constants $c_{+\infty}$ and $c_{-\infty}$ are finite. 

Usually, $c_{+\infty}$ and $c_{-\infty}$ are taken to be $1$ for $c_{+\infty}$ and $-1$ or $0$ for $c_{-\infty}$.
\end{defi}

\begin{remark}\label{rem:ReLUisDisc}
Note that all the activation functions presented in Example \ref{ex:activationfunctions} satisfy this condition. Indeed, even though it goes against the intuition for the ReLU, we can obtain another activation function which is continuous and bounded by combining properly two ReLU.

\end{remark}

\subsection{Variable Lebesgue spaces}\label{subsec:variableLp}

In this subsection, we introduce all the necessary notions and previous results on variable Lebesgue spaces to understand the results on approximation with neural networks in those spaces that will be presented in Section \ref{sec:approxvarLp}. A good reference for these spaces and their application to harmonic analysis is \cite{cruzuribe2013variablelp}.

For simplicity, the measurable space we consider hereafter is $(\Omega, \mathcal{B}, |\cdot|)$, where $\Omega\subseteq \R^d$, $\mathcal{B}$ is the $\sigma$-algebra of measurable sets and $|\cdot|$ is the Lebesgue measure. Other measurable spaces could be considered, but here we restrict to this special case, due to its importance for applications.

The motivation for the study of variable Lebesgue spaces is the following: Consider the real function $g(x):=\frac{1}{\sqrt{|x|}}$, which presents a singularity at $x=0$. This natural function does not belong to any $L^p(\R)$ for $1\leq p\leq +\infty$, because it either grows too fast at the origin or decays too slow at infinity. An idea to solve this matter and find a Lebesgue space to which it could belong would be to split the domain. With this on mind, we can prove $g\in L^1([-1,1])$ and $g\in L^4(\R\setminus [-1,1])$, for instance. The disadvantage of this approach is that for more complex examples, the splitting of the domain becomes increasingly difficult. Then, letting the exponent vary point-wise, we could control in a better way the singularities that the function might present at each point. This leads to the origin of \textit{variable Lebesgue spaces}.
 
\begin{defi}\label{def:exponent function}
An \emph{exponent function} is a measurable function $p:\Omega\to [1,+\infty)$.
\end{defi}

Note that the exponent function always takes a finite value at each point of the domain. It is possible to generalize the concept to take the value $+\infty,$ but this is just a technical issue in which we are not particularly interested here.

\begin{defi}\label{def:normVariableLp}
Given a measurable function $f$ and an exponent function $p$, the \emph{norm} is defined by
$$
\| f\|_{p(\cdot)} := \inf \left\{\lambda>0: \int_{\Omega}\left(\frac{\abs{f(x)}}{\lambda}\right)^{p(x)} \mathrm{d}x\leq 1  \right\}.
$$
\end{defi}

In some occasions, we might denote the previous norm by $\| f\|_{L^{p(\cdot)}(\Omega)}$ whenever it is necessary to emphasize the domain of the functions. There exist in the literature different ways to define the norm for these spaces. The previous expression is obtained through the Luxemburg norm, but it is also possible to use the Amemiya norm, or even employ a different modular than  the one used above, i.e. $\rho(f(x))=\int |f(x)|^{p(x)}$. 

We can now introduce the formal definition of variable Lebesgue spaces as follows.

\begin{defi}\label{def:VariableLp}
Let $p:\Omega\to [1,+\infty)$ be an exponent function and consider the space of functions given by:
$$
\Lp(\Omega):=\{f:\Omega\to \R: f \text{ measurable and } \| f\|_{p(\cdot)}<+\infty\} \, .
$$
$(\Lp(\Omega), \|\cdot\|_{p(\cdot)})$ is a Banach space, which we call \emph{variable Lebesgue space}.
\end{defi}

In particular, if the exponent function is constantly equal to a value $p$, we recover the classical Banach space $(L^p(\Omega), \|\cdot\|_p)$ with $1\leq p<\infty$, showing thus that the previous definition is a generalized version of the so-called Lebesgue spaces. When the domain is not relevant we might denote the previous spaces by $\Lp$.

Next, we present an essential definition which allows to split the behavior of variable Lebesgue spaces into two different cases.

\begin{defi}\label{def:pbounded}
We say that an exponent function $p:\Omega\to [1,+\infty)$ is \emph{bounded} when the essential supremum of $p$ verifies:
$$
\ds p_+:= \esssup_{x\in \Omega} p(x) < +\infty \, ,
$$
i.e. it is uniformly bounded except possibly in a set of zero Lebesgue measure. If $p$ is not bounded, we say that it is \emph{unbounded}.
\end{defi}

The importance of the boundedness of the exponent function resides in the fact that variable Lebesgue spaces with bounded exponent function behave similarly to classical Lebesgue spaces. Moreover, they verify  certain important properties, like the ones presented below, whose proofs can be found in \cite{cruzuribe2013variablelp}.
\begin{prop}\label{prop:LpVariableSeparable}
The variable Lebesgue space $(\Lp,\|\cdot\|_{p(\cdot)})$ is separable if, and only if, the exponent function is bounded.
\end{prop}

Therefore, whenever the exponent $p$ is bounded, the space $\Lp$ is separable. For the next property,  we need to consider the  subspaces of $\Lp$ of compactly supported functions and smooth compactly supported functions, denoted by $\Lp_c$ and $C^\infty_c$, respectively. Now, one can prove the following characterization of bounded exponent for variable Lebesgue spaces, concerning their density in $\Lp$. 

\begin{prop}\label{prop:LpVariableDenseCompact}
 $\left(\Lp_c,\|\cdot\|_{p(\cdot)}\right)$ is dense  in $\Lp$ if, and only if, the exponent function is bounded.

Moreover,  $\left(C^\infty_c,\|\cdot\|_{p(\cdot)}\right)$ is dense  in $\Lp$ if, and only if, the exponent function is bounded.
\end{prop}

Next, we present an essential result concerning the duality of this class of spaces. Before stating it,  we recall that for any exponent function $p:\Omega\to [1,+\infty)$, its \emph{Hölder conjugate} is the exponent function $q$ defined point-wise by
$$
\frac{1}{q(x)}+\frac{1}{p(x)}=1 \, ,
$$
for almost every $x\in \Omega$. Note that if $p(x)=1$, then $q(x)=+\infty$. Although this situation does not follow completely the definition of exponent function presented in these lines, the function $q$ can also be interpreted as an exponent function in a generalized sense.

\begin{prop}\label{prop:LpVariableDual}
The dual space of $(\Lp,\|\cdot\|_{p(\cdot)})$ is $(\Lq,\|\cdot\|_{q(\cdot)})$ where $q$ is point-wise the Hölder conjugate of $p$ if, and only if, the exponent function $p$ is bounded. Furthermore, if $p$ is bounded, there is an equivalence between functionals, $T$, and functions in $\Lq$, $g$, through the following identity
$$
T(f)=\int_\Omega f(x)g(x)\mathrm{d}x \, , \phantom{sdad}\text{ for every } f \in \Lp \, .
$$
\end{prop}

Moreover, still in the case in which the exponent function is bounded, there are also some inclusions between variable Lebesgue spaces whenever the domain is compact.

\begin{prop}\label{prop:variableLpContCompact}
Let $K$ be compact. Then, if we have two bounded exponent functions $p_1$ and $p_2$ verifying that for every $x\in K$ it holds that $p_1(x)\leq p_2(x)$, then
$$
L^{p_2(\cdot)}(K)\subseteq L^{p_1(\cdot)}(K) \, .
$$
\end{prop}

On the other hand, if the exponent function is unbounded, the situation is much more subtle. Variable Lebesgue spaces with unbounded exponent functions are much more complex in general. They are non-separable, we cannot approximate their elements by compactly supported functions  and the dual is more complicated (see previous work of one of the authors in \cite{amenta2019variablelp}). In the following, we present some notions needed to deal with unbounded exponent functions.
\begin{defi}\label{def:LpQ}
 Let $p:\Omega\to [1,+\infty)$ be an unbounded exponent function and consider the quotient space by the closure of the compactly supported functions of $\Lp$, i.e.
$$
\Lp_Q=\Lp/\overline{\Lp_c} \, .
$$
This quotient is equipped with the quotient norm $\|\cdot\|_Q$, which is defined by
$$
\|[f]\|_Q:=\inf_{g\in \Lp_c} \|f-g\|_{p(\cdot)} \, .
$$
Note that when $p$ is unbounded, this is nontrivial (Proposition \ref{prop:LpVariableDenseCompact}).
\end{defi}

This way of computing the quotient norm is impractical. To overcome that issue, in \cite{amenta2019variablelp} the following useful characterization was proven.

\begin{prop}\label{prop:charNormQ}
Let $p: \Omega \rightarrow [1, + \infty)$ be an unbounded exponent function. Then,
$$
\|[f]\|_Q=\inf \left\{\lambda>0: \int_{\Omega}\left(\frac{\abs{f(x)}}{\lambda}\right)^{p(x)} \mathrm{d}x<+\infty \right\} \, .
$$
\end{prop}

Note that this expression has the same flavor than the norm of $\Lp$. Furthermore, it is easy to see that $\|[f]\|_Q\leq \|f\|_{p(\cdot)}$, because of the inclusion of the subsets of $\lambda$.

To conclude this subsection on variable Lebesgue spaces, we include a result which relates $\Lp$ spaces, when $p$ is unbounded, to the space of bounded functions. For that, given $A\subseteq \Omega$, we define the application
$$
w(A)=\|[\id_A ]\|_Q \, .
$$

\begin{prop}\label{prop:LinfContainedLp}
Let $p$ be an unbounded exponent function. Then, $L^\infty(\Omega)\subset \Lp(\Omega)$ if, and only if, $w(\Omega)<+\infty$.
\end{prop}

By virtue of the characterization of the quotient norm, this result essentially states that, when $\Omega$ has infinite Lebesgue measure, if $p$ diverges fast enough, then $w(\Omega)<+\infty$ and $L^\infty(\Omega)\subset \Lp(\Omega)$. Some important examples of exponent functions, when we take $\Omega=[1,+\infty)$, are $p(x)=x$ and $p(x)=[x]$, where $[x]$ denotes the integer part of $x$. These exponent functions verify $w(\Omega)<+\infty$. 

Nevertheless, not every exponent function verifies this condition, and this is something that can be identified from its velocity of divergence. Indeed, for example, $p(x)=1+\log x$ or any other exponent function which diverges faster verifies $w(\Omega)<+\infty$. However, there are other unbounded exponent functions which do not verify the previous condition, such as $p(x)=1+\log(1+\log x)$ or any other exponent function which diverges slower than this.

\section{Universal approximation in function spaces}\label{sec:UA}

The aim of this section is to present some of the known results about the property of \textit{Universal Approximation} (UA) for classical function spaces, as well as some improvements to some of them. A nice survey of these results in soft computing techniques is \cite{tikk2003survey}. We are particularly interested in the ones related to the artificial neural networks (ANN) introduced in the previous section.

 Originally, in the mathematical theory of artificial neural networks, the Universal Approximation Theorem stated that a feedforward neural network with a single hidden layer containing a finite number of neurons can approximate arbitrarily well real-valued continuous functions on compact subsets of $\mathbb{R}^d$. This result was later extended to several other spaces (see \cite{park1991universal}, \cite{hornik1991approximation}, \cite{hornik1989multilayer} \cite{cybenko1989approximation}).

In this section, we focus on results of universal approximation in various function spaces, namely those of continuous functions and spaces of locally integrable functions. For each of these classes, we review some of the previous literature and present some new results, with improvements concerning either the domain of the functions, the domain of the activation functions or the techniques involved to prove them.

\subsection{Space of continuous functions}\label{subsec:UAcont}

In 1989, in both \cite{cybenko1989approximation} and \cite{hornik1989multilayer}, the authors proved simultaneously  for many different activation functions that for any continuous function on a compact set $K$ of $\R^d$ and any $\varepsilon>0$, there exists a feedforward neural network with one hidden layer which uniformly approximates the function with precision $\varepsilon$. Before stating formally the result and proving it, we introduce a concept concerning activation functions that is necessary to understand the statement of the theorem.

\begin{defi}
Given an activation function $\sigma : \mathbb{R} \rightarrow \mathbb{R}$ and a compact $K\subset \R^d$, we say that $\sigma $ is \emph{discriminatory for $C(K)$} if for $\mu$ a finite, signed regular Borel measure on $K$, the fact that the following  holds
\begin{equation*}
\int_K \sigma( w \cdot x + b) \, \operatorname{d} \mu (x) = 0
\end{equation*}
for all $w \in \R^d$ and $b \in \mathbb{R}$ implies that $\mu \equiv 0$.

\end{defi}

Now, we can state the first result on universal approximation using ANN, namely the universal approximation theorem for continuous functions on a compact, whose proof can be found in  \cite[Theorem 1]{cybenko1989approximation} or \cite[Theorem 2.1]{hornik1989multilayer}.

\begin{theorem}\label{thm:UAC(K)}
Let $\sigma$ be any continuous discriminatory for $C(K)$ activation function. Then, the finite sums of the form
$$
\displaystyle \sum_{j=1}^M \alpha_j \, \sigma(w_j\cdot x+b_j)
$$
are dense in $C(K)$, the space of continuous functions over a compact $K$.
\end{theorem}

Note from this result that a natural question is whether an activation function is discriminatory for $C(K)$.  We can provide the following sufficient conditions to check when a continuous  function $\sigma$ is discriminatory for $C(K)$, whose proof also appears in \cite[Lemma 1]{cybenko1989approximation} and \cite[Theorem 5]{hornik1991approximation}.

\begin{prop}\label{prop:SufficientCondDisc}
Any bounded, measurable non-constant function, $\sigma$, is discriminatory for $C(K)$ for every compact $K\subset \R^d$. In particular, any continuous sigmoidal function is discriminatory.
\end{prop}

By virtue of this result, we can check that all the examples of activation functions provided in Example \ref{ex:activationfunctions} are discriminatory for $C(K)$. However, to get a continuous approximation the activation function needs to be continuous; thus, to apply Theorem \ref{thm:UAC(K)} we can only use the examples that are continuous, namely the rectifier, the logistic sigmoid and the hyperbolic tangent. 

As a generalization of the previous result, shortly after it the following characterization for discriminatory was provided in \cite{leshno1993multilayer}.

\begin{prop}\label{prop:EquivalenceCondDisc}
Let $\sigma$ be locally bounded and continuous almost everywhere. Then, it is discriminatory for $C(K)$ for every compact $K\subset \R^d$ if, and only if, it is different from an algebraic polynomial almost everywhere.
\end{prop}

Building on these results, we present now some improvements to some of them. The following is a slight improvement of Theorem \ref{thm:UAC(K)}, in the sense that it extends that result from a compact to the space of  functions which vanish at infinity. As a counterpart, it only works in the unidimensional case.

\begin{prop}\label{prop:DiscriminatoryC0(R)}
Let $\sigma\in C(\R)$ be sigmoidal and  non-constant, and let $f\in C(\R)$ with compact support. Then, for every $ \varepsilon>0$ there is $g_\varepsilon\in H_\sigma$ such that 
$$\ds \sup_{x\in\R}|f(x)-g_\varepsilon(x)|<\varepsilon \, .$$
\end{prop}
\begin{proof}
First, we know that $\sigma$ is discriminatory for $C(K)$ with $K$ compact because it is bounded and non-constant (Proposition \ref{prop:SufficientCondDisc}). 

Next, we recall that $C_0(\R)$ is the subspace of continuous functions that vanishes at infinity. This space coincides with the closure in the supremum norm of the union of $C([-\ell,\ell])$ with $\ell \in \N$, viewed as subspaces of $L^\infty(\R)$. Thus, $\sigma$ is also discriminatory for $C_0(\R)$, since there are continuous functions vanishing at infinity in the subspace $H_\sigma$, due to the fact that $\sigma$ is sigmoidal. 

Therefore, it is dense and we can find $\forall \varepsilon>0$ a function $g_\varepsilon\in H_\sigma$ such that 
\begin{equation}
\ds \sup_{x\in\R}|f(x)-g_\varepsilon(x)|<\varepsilon \, .
\end{equation}
\end{proof}
The following result takes a further step in the approximation of continuous functions, since the domain is now the whole $\mathbb{R}^d$ instead of just a compact $K$. It has the same spirit than \cite[Theorem 2]{park1991universal}, although the latter result concerns radial activation functions and we have considered a different structure slicing with hyperplanes.  Here we check that this result also holds in our setting.

\begin{theorem}\label{thm:UAC(Rn)}
Let $\sigma$ be continuous and discriminatory for $C(K)$ for every $K\subset\R^d$ compact. Given $\varepsilon>0$ and $f\in C(\R^d)$, there is a function $g \in H_\sigma$ such that $d(f,g)<\varepsilon$, where the distance is defined in the following way:
$$
d(f,g):=\sum_{k=1}^\infty 2^{-k}\frac{\|(f-g)\id_{[-k,k]^d}\|_\infty}{1+\|(f-g)\id_{[-k,k]^d}\|_\infty} \, ,
$$
where $[-k,k]^d$ is the centered $d$-dimensional cube of side $2k$.
\end{theorem}

We omit the proof of this result, as it is completely analogous to that of Theorem \ref{thm:UALloc} for locally integrable functions. However, we remark a couple of facts concerning this result.

\begin{remark}\label{rem:C(Rn)Rectifier}
The previous theorem also holds for $\sigma$ the rectifier. Moreover, as we can see in the proof of Theorem \ref{thm:UALloc}, we only use the  assumption of continuity for $f$ to justify that it is bounded over compacts. Then, we could have stated the theorem above for $f\in L^\infty_\text{loc}(\mathbb{R}^d)$.
\end{remark}

\subsection{Negative results}
Next, we explore the opposite direction, i.e. negative results of universal approximation. First, we show that the following is a necessary condition for a space to satisfy the UA.

\begin{lemma}\label{lem:SeparabilityNecces}
If a metric space $(X,d)$ of functions can be universally approximated by finite sums of the form
$$
\displaystyle \sum_{j=1}^M \alpha_j \, \sigma(w_j\cdot x+b_j) \, ,
$$
where $\sigma$ is continuous and sigmoidal, then $(X,d)$ is separable.
\end{lemma}
\begin{proof}
The proof easily follows from the fact that $\mathbb{Q}$ is dense in $\mathbb{R}$. Indeed, due to the uniform continuity of $\sigma$, the countable subspace consisting of the finite sums 
$$
\displaystyle \sum_{j=1}^M \alpha_j \, \sigma(w_j\cdot x+b_j) \, ,
$$
where the variables $\alpha_j, b_j$ and the components of the vector $w_j$ are taken to be rational numbers is a dense subspace.
\end{proof}

Since every sigmoidal continuous activation function is uniformly continuous, Lemma \ref{lem:SeparabilityNecces} states that the size of the function space is an important fact to consider in the problem of UA. Furthermore, it can be used to show negative results of UA, such as the following example:

\begin{ex}\label{ex:LinfnotUA}
$(L^\infty(\R^d),\|\cdot\|_\infty)$ is a non-separable space. Therefore, UA fails for continuous sigmoidal activation functions. Moreover, it even fails for continuous activation functions. 

Indeed, if $\sigma$ is continuous, every finite sum of the form
$$
\displaystyle g(x)=\sum_{j=1}^M \alpha_j \, \sigma(w_j\cdot x+b_j)
$$
is continuous. The Heaviside step function  belongs to $L^\infty(\R)$ and has a jump discontinuity at $x=0$ of size $1$. Therefore, at $x=0$ the approximated continuous function $g$ must be simultaneously close to the value $0$ and $1$. Thus, using the continuity of $g$ we can see easily that
$$
\ds \|f-g\|_\infty:=\sup_{x\in \R}\left|h(x)- \sum_{j=1}^M \alpha_j \, \sigma(w_j\cdot x+b_j)\right|\geq \frac{1}{3} \, ,
$$
concluding thus the proof that for $(L^\infty(\R^d),\|\cdot\|_\infty)$ UA fails for continuous activation functions. 

\end{ex}

\subsection{Locally integrable functions}\label{subsec:locallyint}

The space of locally integrable functions, $L^1_{loc}(\R^d)$, is the largest space that we deal with in this manuscript. It includes as special cases the previous function spaces. However, this space lacks an associated norm. One can solve this problem by constructing a distance in a similar way that we have already done for $C(\R^d)$ in Theorem \ref{thm:UAC(Rn)}.

\begin{theorem}\label{thm:UALloc}
Given $\sigma$ discriminatory for $C(K)$ for every $K\subset \R^d$ compact, $\varepsilon>0$ and $f\in L^1_{loc}(\R^d)$, there is a function $g$ of the form:
$$
g(x) = \displaystyle \sum_{j=1}^M \alpha_j \, \sigma(w_j\cdot x+b_j) \, ,
$$
for every $x \in \mathbb{R}^d$, such that $d(f,g)<\varepsilon$, where the distance is defined in the following way:
$$
d(f,g):=\sum_{k=1}^\infty 2^{-k}\frac{\|(f-g)\id_{[-k,k]^d}\|_1}{1+\|(f-g)\id_{[-k,k]^d}\|_1} \, ,
$$
where $[-k,k]^d$ is the centered $n$-dimensional cube of side $2k$.
\end{theorem}
\begin{proof}
The key idea of the proof relies on the fact that, due to the construction of the distance, the behaviour away from  compact sets of the form $[-k,k]^d$ has very few impact in the computation of the distance, because of the multiplicative factor $2^{-k}$. Therefore, given $\varepsilon>0$, there is a large enough $k$ so that we can restrict to approximating the function $f$ on $[-k,k]^d$. 

 Indeed, fix $f:\mathbb{R}^d \rightarrow \mathbb{R}$ continuous and $\varepsilon > 0$. Take $m \in \mathbb{N}$ such that
\begin{equation*}
\frac{1}{2^m} < \frac{\varepsilon}{2} \, .
\end{equation*}

Hence, this clearly implies that 
\begin{equation*}
\un{k=m+1}{\ov{\infty}{\sum}} \, \frac{1}{2^k} < \frac{\varepsilon}{2} \, .
\end{equation*}

Now, by virtue of Theorem \ref{thm:UAC(K)}, there exists a function $g$ of the form $g(x) = \un{j=1}{\ov{\ell}{\sum}} \alpha_j \, \sigma (w_j \cdot x + b_j) $ for every $x \in \mathbb{R}^d$
such that
\begin{equation*}
\norm{(f-g)\id_{[-m,m]^d}}_1\leq \norm{(f-g)\id_{[-m,m]^d}}_\infty (2m)^d < \frac{\varepsilon}{2} \, .
\end{equation*}

Therefore, it is clear that
\begin{align*}
d(f,g) & = \sum_{k=1}^\infty 2^{-k}\frac{\|(f-g)\id_{[-k,k]^d}\|_1}{1+\|(f-g)\id_{[-k,k]^d}\|_1} \\
& =  \sum_{k=1}^{m} 2^{-k}\frac{\|(f-g)\id_{[-k,k]^d}\|_1}{1+\|(f-g)\id_{[-k,k]^d}\|_1} + \sum_{k=m+1}^{\infty} 2^{-k}\frac{\|(f-g) \id_{[-k,k]^d}\|_1}{1+\|(f-g) \id_{[-k,k]^d}\|_1} \\
& \leq \|(f-g)\id_{[-m,m]^d}\|_1 + \sum_{k=m+1}^{\infty} 2^{-k} \\
& < \frac{\varepsilon}{2} + \frac{\varepsilon}{2} = \varepsilon \, ,
\end{align*}
where we are using in the second line the fact that $f-g$ is locally integrable, so that $\|(f-g) \id_{[-k,k]^d}\|_1$ is finite for every $k \in \mathbb{N}$ and thus 
\begin{equation*}
\frac{\|(f-g) \id_{[-k,k]^d}\|_1}{1+\|(f-g) \id_{[-k,k]^d}\|_1} \leq 1 \; \; \; \text{ for every } k \in \mathbb{N} \, .
\end{equation*}

\end{proof}

\begin{remark}\label{rem:L1locMasSigma}
As in the case of Theorem \ref{thm:UAC(Rn)}, a result in a similar spirit appeared in \cite{park1991universal}, although the latter result concerned radial activation functions. 

\end{remark}

\section{Approximation in variable Lebesgue spaces}\label{sec:approxvarLp}

In the previous section, we have collected some results concerning the universal approximation property with neural networks in certain function spaces. However, to the best of our knowledge, there is no result about universal approximation for variable Lebesgue spaces. 

In this section, we present the main results of this paper. We have split them into three different cases, depending on the exponent function of the space and its domain. First, we show some results of universal approximation whenever the exponent function is bounded, subsequently proving that, under certain conditions on the activation function, this is the only case for which a universal approximation is possible. Later, we shift towards the unbounded exponent function setting, starting with the case in which the domain is discrete and subsequently lifting these results to a general domain case.

\subsection{Case I: Bounded exponent function}\label{subsec:bounded}

In this section, we discuss the results of approximation for variable Lebesgue spaces obtained when the exponent function is bounded, i.e. whenever $p: \Omega \rightarrow [1, + \infty)$ for $\Omega \subseteq \mathbb{R}^d$ verifies
\begin{equation*}
\underset{x \in \Omega}{\text{ess sup}} \, p(x) < + \infty \, .
\end{equation*}

In the case that the exponent function is bounded, the variable Lebesgue space $\Lp$ is separable. Furthermore, smooth functions with compact support are dense. Thus, we can present a first version of universal approximation result for bounded variable Lebesgue spaces relying on the previous fact and Proposition \ref{prop:EquivalenceCondDisc}.

\begin{theorem}\label{thm:UALp}
Let $\sigma: \mathbb{R} \rightarrow \mathbb{R}$ be locally bounded, continuous almost everywhere and different from an algebraic polynomial. Then, truncated finite sums of the form
$$
g(x)=\begin{cases}
\displaystyle \sum_{j=1}^M \alpha_j \, \sigma(w_j\cdot x+b_j) & x\in K \, , \\
0 & x\in \R^d\setminus K \, ,
\end{cases}
$$
 with $K$ compact are dense in $\Lp(\R^d)$ with bounded exponent $p$.
\end{theorem}
\begin{proof}
Fix $\varepsilon>0$. Since $C^\infty_c(\R^d)$ is dense in $\Lp(\R^d)$ by Proposition \ref{prop:LpVariableDenseCompact}, we can find $f_1\in C(\R^d)$ with compact support $K$, such that $\|f- f_1\|_{\Lp(\R^d)}<\varepsilon/2$.

Using Theorem \ref{thm:UAC(K)}, we can find $g_\varepsilon\in H_\sigma$ such that it is truncated to be zero outside $K$ and
$$
\ds \sup_{x\in K}|f_1(x)-g_\varepsilon(x)|<\frac{\varepsilon}{2 |K|} \, .
$$
Therefore, 
$$
\|f-g_\varepsilon\|_{\Lp(\R^d)}\leq \|f-f_1\|_{\Lp(\R^d)}+\|f_1-g_\varepsilon\|_{\Lp(\R^d)}<\frac{\varepsilon}{2}+|K|\frac{\varepsilon}{2 |K|}=\varepsilon.
$$
\end{proof}

Furthermore, when the exponent function is bounded, the dual of variable Lebesgue spaces is fully characterized (see Proposition \ref{prop:LpVariableDual}) and it coincides with $\Lq$, where $q$ is point-wise the Hölder conjugate of $p$, i.e. $\frac{1}{p(x)}+\frac{1}{q(x)}=1$ for every $x \in \Omega$. By virtue of this result, and analogously to the previous notions of discriminatory activation functions, we can give the following definition:
\begin{defi}
Let $\sigma : \mathbb{R} \rightarrow \mathbb{R}$ be an activation function and $K\subset \R^d$ compact. For a bounded exponent function $p: \Omega \rightarrow [1, + \infty)$, for $\Omega \subseteq \mathbb{R}^d$, we say that $\sigma $ is \emph{discriminatory for $\Lp(K)$} if for any $h \in \Lq(K)$, where $q$ is pointwise the Hölder conjugate of $p$ (i.e. $1/p(x)+1/q(x)=1$ for every $x \in \Omega$), whenever
\begin{equation*}
\int_K \sigma( w \cdot  x + b) \, h(x) \, d x = 0
\end{equation*}
for all $ w \in \R^d$ and $b \in \mathbb{R}$ implies that $h = 0$ almost everywhere.
\end{defi}

 As mentioned above, this is indeed a reasonable definition due to the fact that $\Lq$ can be identified with the dual of $\Lp$. Now, considering this definition, we can state and prove the following result in this setting.

\begin{theorem}\label{thm:UALpVariable}
Let $\Omega\subseteq \R^d$, $p: \Omega \rightarrow [1, + \infty)$ a bounded exponent function and $\sigma: \mathbb{R} \rightarrow \mathbb{R}$  discriminatory for $\Lp(K)$ for every compact $K\subset \Omega$. Then, truncated finite sums of the form
$$
g(x)=\begin{cases}
\displaystyle \sum_{j=1}^M \alpha_j \, \sigma(w_j\cdot x+b_j) & x\in K \, , \\
0 & x\in \Omega \setminus K \, ,
\end{cases}
$$
 with $K$ compact  are dense in $\Lp(\Omega)$.
\end{theorem}
\begin{proof}
Given $f\in \Lp(\Omega)$ and $\varepsilon>0$, we apply Proposition \ref{prop:LpVariableDenseCompact} to find a compact $K$  where 
$$
\|f-f \id_K\|_{L^p(\cdot)(\Omega)} < \frac{\varepsilon}{2} \, .
$$

Then, we can show that there is a $g$ of the form $\displaystyle \sum_{j=1}^d \alpha_j \, \sigma(w_j\cdot x+b_j)$ such that 
$$
\displaystyle \left\|f- \sum_{j=1}^M \alpha_j \, \sigma(w_j\cdot x+b_j) \right\|_{\Lp(K)}<\frac{\varepsilon}{2} \, ,
$$
finishing thus the proof. Indeed, if above were false, we would get a contradiction  as a consequence of Hahn-Banach theorem. By virtue of its version as a hyperplane separation theorem (see \cite[Theorem 3.5]{rudin1921analysis}),  we can find a nontrivial functional, $T$, which vanishes over the all the finite sums of the form
$$
\displaystyle \sum_{j=1}^M \alpha_j \, \sigma(w_j\cdot x+b_j) \, .
$$

From Proposition \ref{prop:LpVariableDual}, we know that the dual of $\Lp(K)$ is $\Lq(K)$ where $q$ is point-wise the Hölder conjugate of $p$ and
$$
T(f)=\int_K f(x) h(x) \mathrm{d}x \, .
$$

Therefore, we have obtained a relation between functionals, $T$, and elements of $\Lq(K)$, $h$.  Since $\sigma$ is discriminatory, and $h$ can in particular be taken to be $\sigma$ above (as $\sigma$ belongs to $\Lq(K)$), we conclude $T=0$, getting a contradiction with the nontriviality of $T$.
\end{proof}

A natural question that yields the previous result is when the activation function is discriminatory for variable Lebesgue spaces. We can hence prove the following relation between the properties of being discriminatory, which in particular yields the fact that Theorem \ref{thm:UALpVariable} is more general than Theorem \ref{thm:UALp}.

\begin{lemma}\label{lem:RelationDiscVar}
Given $K\subset \R^d$ compact and two bounded exponent functions $p_1$ and $p_2$ satisfying $p_1(x)\leq p_2(x)$ for all $x\in K$,
$$
\{\sigma \text{ disc. for } C(K)\} \subseteq \{\sigma \text{ disc. for } L^{p_2(\cdot)}(K)\}\subseteq \{\sigma \text{ disc. for } L^{p_1(\cdot)}(K)\} \, .
$$

In particular, non-constant, bounded activation functions $\sigma$ and the rectifier are discriminatory for $\Lp(K)$ (with $p$ bounded).
\end{lemma}
\begin{proof}
We have the following inclusion of variable Lebesgue space when the domain is compact: Whenever $p_1\leq p_2$, it holds that $L^{p_2(\cdot)}(K)\subseteq L^{p_1(\cdot)}(K)$ (Proposition \ref{prop:variableLpContCompact}). We recall that the dual of bounded variable Lebesgue space $\Lp$ is the variable Lebesgue space $\Lq$ where $q$ is pointwise the Hölder conjugate of $p$ (Proposition \ref{prop:LpVariableDual}). Therefore,
$$
L^{q_1(\cdot)}(K)\subseteq L^{q_2(\cdot)}(K)\, ,
$$
where $q_1$ and $q_2$ are the Hölder conjugates of $p_1$ and $p_2$, respectively. This proves that a function $\sigma$ that is discriminatory for $L^{p_2(\cdot)}(K)$ is also discriminatory for $L^{p_1(\cdot)}(K)$.

Moreover, since for every $h\in L^1(K)$, $h(x)\mathrm{d}x$ is a Radon measure, we have if $\sigma$ is discriminatory for $C(K)$, it is also discriminatory for $\Lp(K)$ with $p$ bounded.

Finally, using Proposition \ref{prop:SufficientCondDisc} and Remark \ref{rem:ReLUisDisc},  non-constant, bounded activation functions $\sigma$ and the rectifier are discriminatory for $\Lp(K)$ (with $p$ bounded).
\end{proof}

From Lemma \ref{lem:RelationDiscVar} and Proposition \ref{prop:SufficientCondDisc}, it easily follows that all the examples shown in Subsection \ref{subsec:neuralnetworks} are discriminatory for $\Lp(K)$. Therefore, we can use Theorem \ref{thm:UALpVariable} with any of the activation functions from Example \ref{ex:activationfunctions}.

We conclude this subsection showing the connection between the boundedness of the exponent function and the universal approximation property for $\Lp$ spaces.
\begin{cor}\label{cor:characterization}
Let $\Omega\subseteq \R^d$, $p: \Omega \rightarrow [1, + \infty)$ a exponent function and $\sigma: \mathbb{R} \rightarrow \mathbb{R}$  continuous and sigmoidal. Then, UA holds for $\Lp(\Omega)$ if, and only if, $p$ is bounded.
\end{cor}
\begin{proof}
On the one hand, from Theorem \ref{thm:UALpVariable}, Proposition \ref{prop:SufficientCondDisc} and Lemma \ref{lem:RelationDiscVar} we deduce that, when $p$ is bounded, we have UA for $\Lp(\Omega)$. On the other hand, when $p$ is unbounded the space $\Lp(\Omega)$ lacks UA, due to Proposition \ref{prop:LpVariableSeparable} and Lemma \ref{lem:SeparabilityNecces}.
\end{proof}

\subsection{Case II: Unbounded exponent function, discrete case}\label{subsec:unboundeddiscrete}

After having dealt with the bounded case in the last section, we now shift towards the unbounded case. For that, it is better to show first the results that can be obtained in the discrete case, i.e. in variable sequence spaces. Even though variable sequence spaces are easier to describe than variable Lebesgue spaces, they are complex enough to show that it is impossible to achieve a universal approximation property.

More precisely, we consider the measure space $(\N, 2^{\N}, \mu)$, where $2^{\N}$ denotes all the subsets of the natural numbers and $\mu$ is the counting measure (it associates to every subset of the natural numbers its cardinality). As in the continuum case, other measures could be considered and the results would follow with minor modifications.

Let us recall the definition of a variable sequence space: Given an exponent function $p:\N\to [1,+\infty)$, consider the modular
$$
\rho_{p(\cdot)}(\{x(k)\}):=\sum_{j=1}^\infty |x(j)|^{p(j)} \, .
$$

Then, the norm is given by
$$
\|\{x(k)\}\|_{p(\cdot)}=\inf\left\{\lambda>0: \rho_{p(\cdot)}(\{x(k)\}/\lambda)\leq 1 \right\}.
$$

Let us recall that we are denoting the subspace of all sequences that can be obtained using an ANN with 1 hidden layer and activation function $\sigma$ by
$$
H_\sigma :=\left\{\{y(k)\}: y(k)=\sum_{j=1}^M  \alpha_j \, \sigma(w_j\cdot k+b_j)\right\} \, ,
$$
where $\alpha_j, w_j, b_j\in \R$ and $M\in \N$. Hereafter, we will assume that the activation function satisfies $\sigma \in L^\infty(\R)$ non-constant and sigmoidal (existence of the limits at $+\infty$ and $-\infty$). Therefore, $H_\sigma \subseteq \ell^\infty$. More specifically, since $\sigma$ is sigmoidal, $H_\sigma$ is a subspace of the space of convergent sequences, i.e. $H_\sigma \subseteq c$.

Before proceeding to the main results of this section, here we collect some results concerning variable sequence spaces and its relation with $\ell^\infty$ addressed in \cite{amenta2019variablelp}.
\begin{prop}
Let $p:\N\to [1,+\infty)$ be an exponent function. Then,
\begin{itemize}
\item $\ell^{p(\cdot)}\subseteq \ell^\infty$.
\item $\ell^{p(\cdot)} = \ell^\infty$ as vector spaces if, and only if, $\|\id_{\N}\|_{p(\cdot)}<+\infty$.
\end{itemize}
\end{prop}

The condition $\|\id_{\N}\|_{p(\cdot)}<+\infty$ is related to the divergence of the exponent function $p$, that is, we need that $p(k)\to +\infty$ ($k\to+\infty$) fast enough so that the series
$$
\sum_{j=1}^\infty s^{p(j)}
$$
converges for some $s>0$. For example, $p(k)=1+\log k$ or any other exponent function which diverges faster verifies the previous condition. However, there are other unbounded exponent functions which do not verify the previous condition, such as $p(k)=1+\log(1+\log k)$ or any other exponent function which diverges slower than this.

\begin{prop}\label{prop:ApproxVariableSequenceSp}
Let $p:\N\to [1,+\infty)$ be an exponent function such that $\|\id_{\N}\|_{p(\cdot)}<+\infty$ and $\sigma\in L^\infty(\R)$ sigmoidal and non-constant.  Then, $\overline{H_\sigma}=c$, where the closure is taken in the $\|\cdot\|_{p(\cdot)}$ norm. 
\end{prop}
\begin{proof}
Note that using the hypothesis of the exponent function, $(\ell^{p(\cdot)}, \|\cdot\|_{p(\cdot)})$ and $(\ell^{\infty}, \|\cdot\|_{\infty})$ is the same space with two equivalent norms. 

On the one hand, since $c$ is closed in $(\ell^{\infty}, \|\cdot\|_{\infty})$, from the obvious inclusion $H_\sigma\subseteq c$ we get that $\overline{H_\sigma}\subseteq c$.

On the other hand, to prove the reverse inclusion it is enough to show that the space of sequences with finite number of nonvanishing terms $c_{00}$ is contained in $\overline{H_\sigma}$, since the closure in $(\ell^{\infty}, \|\cdot\|_{\infty})$ of $c_{00}$ is $c_0$ (the space of sequences which converges to $0$) and we have good approximations of constants in $H_\sigma$ because $\sigma$ is sigmoidal.

Finally, $c_{00}$ can be approximated by a an analogous discrete argument as Proposition \ref{prop:DiscriminatoryC0(R)}.
\end{proof}

\begin{remark}
The statement of Proposition \ref{prop:ApproxVariableSequenceSp} holds for $\sigma$  the rectifier. More precisely, we can prove that $\overline{H_\sigma \cap \ell^{p(\cdot)}}=c$, since using two ReLU we can obtain a non-constant, sigmoidal, bounded activation function and appeal then to Proposition \ref{prop:ApproxVariableSequenceSp}.
\end{remark}

\subsection{Case III: Unbounded exponent function, general case}\label{subsec:unboundedgeneral}

Now we analyze the general situation in which the exponent function is unbounded. In this case, the variable Lebesgue space is nonseparable (see Proposition \ref{prop:LpVariableSeparable}). Using Lemma \ref{lem:SeparabilityNecces}, we deduce that the  separability of the function space is necessary for a Universal Approximation result to hold for the sigmoidal, continuous activation function. Therefore, in the current context, the difficulty to obtain approximation results is much higher. For example, note that we cannot even approximate the function by its restriction to compact domains (Proposition \ref{prop:LpVariableDenseCompact}). 

Moreover, there is an additional problem associated to finding a precise characterization of the dual, since many subtleties appear in this setting. For more information about the dual in this case, we refer the interested reader to some previous work of one of the authors \cite{amenta2019variablelp}.

For the reasons aforementioned, here we just focus in the unidimensional case, having on mind the toy model $\Omega=[1,+\infty)$ and an exponent function $p: \Omega \rightarrow \Omega$ given by $p(x)=x$ or $p(x)=[x]$, where $[x]$ denotes the integer part of $x$. In this case, we show a characterization of the subspace of functions which actually can be approximated using an artificial neural network. 

First, since we are focusing in the unidimensional case, we start by showing that we can approximate bounded functions with limit at $\infty$.

\begin{prop}\label{prop:ApproxBoundedLimit}
Let $\Omega\subseteq \R$ be an unbounded interval and $p: \Omega \rightarrow [1, + \infty)$ an unbounded exponent function such that $L^\infty(\Omega)\subset \Lp(\Omega)$ and it is bounded in every compact subset of $\Omega$. Let $\sigma\in L^\infty(\R)$ be a non-constant, sigmoidal activation function. Then, given $f\in L^\infty(\mathbb{R})$ with limit at $\infty$, and $\varepsilon>0$, there is a function of the form
$$
g_\varepsilon(x):=\sum_{j=1}^n \alpha_j \,\sigma(w_j\cdot x+b_j)
$$
such that $\|f-g_\varepsilon\|_{\Lp(\Omega)}<\varepsilon$.
\end{prop}
\begin{proof}
Without loss of generality we can assume that $\Omega=[1,+\infty)$.
Let $\ds \beta:=\lim_{x\to +\infty}f(x)$. For simplicity, we take $\beta=0$. Indeed, since $\sigma$ is sigmoidal, there is function $h(x)=\alpha \, \sigma(w \cdot x+b)$ such that $f-h$ is a bounded function with limit $0$ at $+\infty$, for suitable $\alpha$ and $w$.

Fix $\varepsilon > 0$. Given $\delta>0$ there is $L>0$ such that $|f(x)|<\delta$ if $x>L$. Since $L^\infty(\Omega)\subset\Lp$, we know that $w(\Omega)<+\infty$. Hence, we can take 
$$\delta=\frac{\varepsilon}{4 \|\id_\Omega\|_{\Lp(\Omega)}} \, .$$

Now, since $p$ is bounded in $[1,M]$ for some $M>L$, we can use Theorem \ref{thm:UALpVariable}  and Proposition \ref{prop:DiscriminatoryC0(R)} to find $g_\varepsilon$ of the form
$$
g_\varepsilon(x):=\sum_{j=1}^n \alpha_j \sigma(w_j\cdot x+b_j) \, ,
$$
such that the following holds
$$\|f-g_\varepsilon\|_{\Lp([1,M])}<\frac{\varepsilon}{2} \, ,$$
and $|g_\varepsilon(x)|<\delta$ for $x>M$. Then,
$$
\|f-g_\varepsilon\|_{\Lp(\Omega)}\leq \|f-g_\varepsilon\|_{\Lp([1,M])}+ \|f-g_\varepsilon\|_{\Lp([M,+\infty))}<\frac{\varepsilon}{2}+ 2\delta \|\id_\Omega\|_{\Lp(\Omega)}=\varepsilon
$$ 
\end{proof}

Next, we can proceed to the main result of this paper. In the theorem below, we provide a  characterization of the set of functions of a variable Lebesgue space with an unbounded exponent that can be approximated using neural networks.

\begin{theorem}\label{thm:UApnotbounded}
Let $\Omega\subseteq \R$ be an unbounded interval and $p: \Omega \rightarrow [1, + \infty) $ be an unbounded exponent function such that $L^\infty(\Omega)\subset \Lp(\Omega)$ and it is bounded in every compact subset of $\Omega$. Let $\sigma\in L^\infty(\R)$ be a non-constant, sigmoidal activation function. Then, the following conditions are equivalent for $f\in \Lp(\Omega)$:
\begin{enumerate}
\item For every $ \varepsilon>0$, there is a function of the form
$$
g_\varepsilon(x):=\sum_{j=1}^n \alpha_j \sigma(w_j\cdot x+b_j) \, ,
$$
such that $\|f-g_\varepsilon\|_{\Lp(\Omega)}<\varepsilon$.
\item There is a scalar $\beta \in \R$ such that
$$
\|[f-\beta \id_{\Omega}]\|_Q=0 \, ,
$$
where $\|\cdot\|_Q$ is the quotient norm given in Definition \ref{def:LpQ}.
\end{enumerate}
\end{theorem}

\begin{remark}
The rectifier in Example \ref{ex:activationfunctions} is also a valid activation function for Theorem \ref{thm:UApnotbounded}, since we can obtain a continuous, bounded activation function as a combination of two ReLU. 
\end{remark}
\begin{proof}[Proof of Theorem \ref{thm:UApnotbounded}]
Without loss of generality we can assume that $\Omega=[1,+\infty)$. First we show that \underline{$1\Rightarrow 2$}. We recall that, since $\sigma$ is sigmoidal (see Definition \ref{def:sigmoidal}), then 
\begin{equation*}
\sigma (t)=\begin{cases} 
       c_{+\infty}  & \text{as } t \rightarrow + \infty \, , \\
       c_{-\infty}  & \text{as }  t \rightarrow - \infty \, .
   \end{cases}
\end{equation*}

Therefore, every function of the form
$$
g(x)=\sum_{j=1}^n \alpha_j \sigma(w_j\cdot x+b_j)
$$
converges to 
$$\ds \sum_{\{j: w_j>0\}} \alpha_j c_{+\infty}+\sum_{\{j: w_j<0\}} \alpha_j c_{-\infty}+ \sum_{\{j: w_j=0\}} \alpha_j \sigma(b_j)$$
when $x$ tends to $+\infty$.

From hypothesis $1$, we know that $\forall \varepsilon>0$, there is a function $g_\varepsilon$ with the previous form such that $\|f-g_\varepsilon\|_{\Lp(\Omega)}<\varepsilon$. Let us denote by $\beta_\varepsilon$ the limit at $+ \infty$ of $g_\varepsilon$. First, we can show that $g_\varepsilon$ and $\beta_\varepsilon \id_\Omega$ belong to the same class in the quotient space.

Indeed, since $\beta_\varepsilon$ is the limit of $g_\varepsilon(x)$ at $+\infty$, for every $\delta>0$ we can find $M>0$ such that for every $x>M$, it holds that $|g_\varepsilon(x)-\beta_\varepsilon|<\delta.$ Then,
$$
\|[g_\varepsilon-\beta_\varepsilon \id_\Omega]\|_Q=\|[(g_\varepsilon-\beta_\varepsilon) \id_{(M,+\infty)}]\|_Q\leq \|[\delta\id_{(M,+\infty)}]\|_Q=\delta w(\Omega),
$$
where the first equality follows from the fact that the quotient space is taken over the closure of the compactly supported functions in $\Lp$. Note that $w(\Omega)$ is finite. Then, as $\delta>0$ can be arbitrary small, we conclude $\|[g_\varepsilon-\beta_\varepsilon \id_\Omega]\|_Q=0$. 

Since these two functions are equivalent in the quotient space, we can deduce the following:
$$
\|[f-\beta_{\varepsilon}\id_\Omega]\|_Q=\|[f-g_{\varepsilon}]\|_Q\leq \|f-g_{\varepsilon}\|_{\Lp(\Omega)}<\varepsilon \, .
$$

Next, we claim that the sequence $\{\beta_\varepsilon\}$ is bounded. Indeed, this holds due to the fact that:
$$
\left|\beta_{\varepsilon_1}-\beta_{\varepsilon_2}\right| w(\Omega)=\|[(\beta_{\varepsilon_1}-\beta_{\varepsilon_2})\id_\Omega]\|_Q\leq \|[\beta_{\varepsilon_1}\id_\Omega-f]\|_Q+\|[f-\beta_{\varepsilon_2}\id_\Omega]\|_Q<\varepsilon_1+\varepsilon_2 \, ,
$$
and thus,
$$
\left|\beta_{\varepsilon_1}-\beta_{\varepsilon_2}\right| \leq \frac{\varepsilon_1+\varepsilon_2}{w(\Omega)} \, .
$$

Since $\{\beta_\varepsilon\}$ is bounded, there is subsequence $\{\beta_{\varepsilon_n}\}$, with $\varepsilon_n$ tending to $0$ and $\beta_{\varepsilon_n}$ converging to some scalar $\beta$ when $n$ tends to $+\infty$. To conclude, we now have to show that 
$$
\|[f-\beta \id_{\Omega}]\|_Q=0 \, ,
$$
finishing thus the proof of $2$. Indeed:
$$
\|[f-\beta \id_{\Omega}]\|_Q\leq \|[f-\beta_{\varepsilon_n} \id_{\Omega}]\|_Q+\|[(\beta_{\varepsilon_n}-\beta) \id_{\Omega}]\|_Q\leq \varepsilon_n+|\beta_{\varepsilon_n}-\beta|w(\Omega) \, .
$$
which clearly vanishes when $n$ tends to $+\infty$.

Now we prove \underline{$2\Rightarrow 1$}. Fix $\varepsilon>0$. From hypothesis $2$, we can find $f_c \in \Lp_c$ such that
$$
\|(f-\beta \id_{\Omega})-f_c\|_{\Lp(\Omega)} <\frac{\varepsilon}{3} \, .
$$

The support of $f_c$ is contained in a compact of the form $[1,L]$ for some $L$. Since $p$ is bounded at $[1,L]$ by hypothesis,  using Theorem \ref{thm:UALpVariable} we can find a function $g_1$ of the form 
$$
g_1=\sum_{j=1}^n \alpha_j \sigma(w_j\cdot x+b_j) \, ,
$$
such that
$$
\|f_c -g_1\|_{\Lp([1,L])} <\frac{\varepsilon}{3} \, .
$$

Since $\beta \id_{\Omega}-g_1 \id_{(l,+\infty)}$ is clearly a bounded function whose limit exits (and is finite) at $+\infty$, we can appeal to Proposition \ref{prop:ApproxBoundedLimit} to find a function 
function $g_2$ of the form 
$$
g_2=\sum_{j=n+1}^{n+m} \alpha_j \sigma(w_j\cdot x+b_j) \, ,
$$
such that
$$
\|(\beta \id_{\Omega}-g_1 \id_{(l,+\infty)}) -g_2\|_{\Lp(\Omega)} <\frac{\varepsilon}{3}\, .
$$

Then, we can take $g_\varepsilon$ in the statement of the theorem to be:
$$
g_\varepsilon(x):=g_1(x)+g_2(x)=\sum_{j=1}^{n+m} \alpha_j \sigma(w_j\cdot x+b_j) \, ,
$$
since
\begin{align*}
\|f-g_\varepsilon\|_{\Lp(\Omega)} &\leq \|(f-\beta \id_{\Omega})-f_c\|_{\Lp(\Omega)}  + \|f_c -g_1\|_{\Lp([1,L])} \\
& \; \; \; \; +\|(\beta \id_{\Omega}-g_1 \id_{(L,+\infty)}) -g_2\|_{\Lp(\Omega)}\\
&<\frac{\varepsilon}{3}+\frac{\varepsilon}{3}+\frac{\varepsilon}{3}=\varepsilon \, .
\end{align*}

\end{proof}

\vspace{0.2cm}

\noindent {\it Acknowledgments.} 
The authors want to thank Wen-Liang Hwang for his valuable comments on an earlier draft and JO would like to thank both him and the Institute of Information Science in Taipei for the hospitality during his research stay in the summer of 2018, when the idea for this project was born. AC acknowledges funding by the Deutsche Forschungsgemeinschaft (DFG, German Research Foundation) under Germanys Excellence Strategy EXC-2111 390814868. JO is partially supported by the grant MTM2017-83496-P from the Spanish Ministry of Economy and Competitiveness and through the “Severo Ochoa Programme for Centres of Excellence in R\&D” (SEV-2015-0554). This project has received funding from the European Union’s Horizon 2020 research and innovation programme under the Marie Skłodowska-Curie grant agreement No 777822.

\bibliographystyle{alpha}
\bibliography{biblio}

\vspace{1cm}
\end{document}